\documentclass[oneside,reqno,11pt]{amsart}
\usepackage{latexsym,amssymb,amsmath,graphics,color}

\theoremstyle{plain}
\newtheorem{theorem}{Theorem}[section]
\newtheorem{lemma}[theorem]{Lemma}

\theoremstyle{definition}

\theoremstyle{remark}

\def\dd{{\mathcal{D}}}

\def\E{{\mathbb{E}}}
\def\F{{\mathcal{F}}}

\def\N{{\mathbb{N}}}
\def\P{{\mathbb{P}}}

\def\R{{\mathbb{R}}}

\def\Y{{\mathbb{Y}}}
\def\Z{{\mathbb{Z}}}
\def\PP{{\mathbb{P}}}
\def\Xa{{\mathcal{X}}}

\newcommand{\s}{\sigma}

\def\RR{\mathbb{R}}
\def\8{\infty}
\def\N{\mathbb{N}}
\def\E{\mathbb{E}}
\def\P{\mathbb{P}}

\def\<{\langle}
\def\>{\rangle}

\newcommand{\jd}{1\slash 2}

\renewcommand{\d}{\delta}
\renewcommand{\a}{\alpha}
\renewcommand{\b}{\beta}

\renewcommand{\l}{\lambda}

\newcommand{\eps}{\varepsilon}

\newcommand{\g}{\gamma}

\newcommand{\supp}{\mathrm{supp}}
\newcommand{\Ind}{\mathbf{1}}

\newtheorem{thm}[equation]{Theorem}
\newtheorem{rem}[equation]{Remark}
\newtheorem{cor}[equation]{Corollary}

\theoremstyle{definition}

\numberwithin{equation}{section}

\begin{document}
	
	
\title{Affine stochastic equation with triangular matrices}	
	\author{
		E. Damek and J. Zienkiewicz}

\begin{abstract}
We study solution $\Xa $ of the stochastic equation
$$
\Xa = A\Xa +B,$$
where $A$ is a random matrix and $B,X$ are random vectors, the law of $(A,B)$ is given and $X$ is independent of $(A,B)$. The equation is meant in law, the matrix $A$ is $2\times 2$ upper triangular,  $A_{11}=A_{22}>0$, $A_{12}\in \R $. A sharp asymptotics of the tail of $\Xa =(\Xa _1,\Xa _2)$ is obtained. We show that under ``so called'' Kesten-Goldie conditions $\P (\Xa _2>t)\sim t^{-\a}$ and   $\P (\Xa _1>t )\sim t^{-\a}(\log t)^{\tilde \a}$, where $\tilde \a =\a$ or $\a \slash 2$.
\end{abstract}
	\maketitle




\bigskip
 Instytut Matematyczny, Uniwersytet Wroclawski,
		 50-384 Wroclaw,
		pl. Grunwaldzki 2/4 Poland, edamek@math.uni.wroc.pl 
		
\section{Introduction}
We consider the stochastic recurrence equation 
\begin{equation}\label{rec} 
X_n=A_nX_{n-1}+B_{n} \quad n\in \N,
\end{equation}
where $(A_n, B_n)$ is an i.i.d.\ sequence, $A_n$ are $d\times d$
 matrices, $B_n$ are vectors and $X_0$ is an initial distribution
 independent of the sequence $(A_n,B_n)$. 
 Under mild contractivity hypotheses (see \cite{BP, B})
the sequence $X_n$ converges in law to a random
variable $\Xa$ satisfying 
\begin{equation}
\label{dif recurrence} \Xa \stackrel{d}{=} A \Xa+B, \end{equation} where 
$(A,B)$ is a generic element of the sequence $(A_n,B_n)$ and $\Xa$ is independent of $(A,B)$. The law of  $\Xa $ is the unique solution of \eqref{dif recurrence}.

There is considerable interest in studying various aspects of the
iteration \eqref{rec} and, in particular, the tail behavior of
$\Xa $. The story started with Kesten \cite{K} who
obtained fundamental results about tails of $\Xa$ for 
matrices $A$ having nonnegative entries with the assumption that there is $n$ such that the product $A_1\cdots A_n$ has strictly positive entries with positive probability.

Given $x=(x_1,\dots x_d)$ in the unit sphere $\mathbb{S}^{d-1}$, let
\begin{equation*}
\langle x, \Xa \rangle =\sum _{j=1}^dx_j\Xa _j,\quad \Xa =(\Xa _1,\dots \Xa _d).
\end{equation*}
Under appropriate assumptions Kesten \cite{K}
proved that there is $\a >0$ and a function $e_{\a }$ on $\mathbb{S}^{d-1}$ such that
\begin{equation}\label{rege}
\lim _{t\to \8 }t^{\a }\P (\langle x,\Xa \rangle >t)=e_{\a }(x),\quad x\in \mathbb{S}^{d-1}
\end{equation}
and $e_{\a }(x) > 0$ for $ x\in \mathbb{S}^{d-1}\cap [0,\8
)^d$. Later on analogous results were obtained by
Alsmeyer and Mentemeier \cite{AM} (invertible
matrices $A$ with density assumptions), Buraczewski at al.
\cite{BDGHU} (similarities), Guivarch and Le
Page \cite{GL1} (matrices satisfying some geometric irreducibility properties but with a possibly singular law), Kl\"uppelberg and Pergamenchtchikov \cite{KP} (random coefficient autoregressive model), Mirek \cite{Mi} (multidimensional Lipschitz recursions). Basic moment assumptions on $A$ and $B$ are such that the asymptotics \eqref{rege} is mainly determined by $A$. See \cite{BDM} for an elementary
explanation of Kesten's result and the other results.

For all the matrices considered above we have the same
tail behavior in all directions, one of the reasons being a certain
irreducibility or homogeneity of the action of the group $\langle \supp A \rangle $ generated by
the support of the law of $A$. The latter is discussed carefully in Section 4.4 of \cite{BDM}. Upper triangular matrices do not fit into any of the frameworks mentioned above. In particular, there are plenty of eigenspaces for the action of $\langle \supp A \rangle $. 

If $A=diag (A_{11},\dots A_{dd})$ is diagonal, $\E
A_{ii}^{\a _i}=1$ and $\a _1, \dots \a _d$ are different (see
e.g.\ \cite{buraczewski:damek:2010},
\cite{BDGHU} and 
\cite[Appendix D]{BDM}) then $\P (\Xa _i>t) \sim t^{-\a _i}$.
To have a more illuminating example consider $2\times 2$ upper triangular matrices ($\P (A_{21}=0)=1$). Suppose that  $A_{11}, A_{22}>0$ and $P(A_{12}\neq 0)>0$,
  $\E  A_{ii}^{\a _i}=1$ and $\a _1\neq \a _2$. If $\a _1<\a _2$ then 
  $$\P (\Xa _i>t) \sim t^{-\a _i},$$
  where $\sim$ means $\lim _{t\to \8}\P (\Xa _i>t) t^{\a _i}$ exists, it is finite and strictly positive. But if $\a _1>\a _2$ then \eqref{rege} holds with $\a _2$, see \cite{DMS}.

The pattern is more general. For $d\times d$ upper triangular matrices such that $\E A_{ii}^{\a _i}=1$, with $\a _i \neq \a _j$, $i\neq j$, we have
\begin{equation*}
\P (\Xa _i>t) \sim t^{-\tilde \a _i},
\end{equation*}
where $\tilde \a _i$ depends on $\a _i,\dots \a _d$ \cite{matsui:swiatkowski:2017}.

There is a natural question what happens when $\a _1=\a _2$. It is addressed in the present paper under the additional assumption that $A_{11}=A_{22}>0$, $A_{12}\in \R$. We observe behavior that has not been observed yet for matrix recursions under ``Kesten-Goldie'' moment assumptions $\E A_{ii}^{\a }=1$, $\E A_{ii}^{\a }\log A_{ii}<\8$, $\E |B|^{\a }<\8 $:
\begin{equation}\label{bothtail}
\P (\Xa _1>t)\sim \begin{cases} (\log t)^{\a }t^{-\a }, \quad \mbox\ if \ \ \E A_{12}A_{11}^{\a -1}\neq 0\\ (\log t)^{\a \slash 2}t^{-\a }, \quad \mbox\ if \ \ \E A_{12}A_{11}^{\a -1}= 0.\end{cases}\end{equation}
For the first time a non trivial slowly varying function appears not as a result of the heavy behavior of $B$ (like in \cite{Gre94, DK}) or weaker assumptions on $A$, see \cite{Kev}, namely. To obtain \eqref{bothtail} we need to study a ``perturbed'' perpetuity, see \eqref{ikszero}, which itself is interesting. $\Xa _2$ satisfies a one dimensional version of \eqref{dif recurrence} and so $\P (\Xa _2>t)\sim t^{-\a }$.   
 
 It turns out that the appearance of triangular matrices in \eqref{rec} generates a lot of technical complications,  it is challenging and it is far from being solved in arbitrary dimension when some $\a_1,\dots \a _d$ may be equal.  
The natural conjecture is 
\begin{equation*}
\P (\Xa _i>t) \sim L_i(t)t^{-\tilde \a _i},
\end{equation*}
where $L_i(t)$ are slowly varying functions, most likely $L_i(t)=(\log t)^{\beta _i}$. 
Even for $2\times 2$ matrices, the case when $\a _1= \a _2$ and $A_{11}, A_{22}$ are different seems to be, in our opinion, out of reach in full generality at the moment. 

Our results apply to the squared volatility sequence $W _t=(\s ^2_{1,t}, \s
^2_{2,t})$ of the bivariate GARCH(1,1) financial model, see \cite{BDM}, Section 4.4.5 and \cite{DMS}.
Then $W_t$ satisfies \eqref{rec} with matrices
$A_t$ having non-negative entries. If all the entries of $A_t$ are
strictly positive then the theorem of Kesten applies and both $\s
^2_{1,t}$ and $\s ^2_{2,t}$ are regularly varying with the same
index, see \cite{matsui:mikosch:2016}, \cite{mikosch:starica:2000}. But if this is not the case then we have to go beyond Kesten's approach as it is done in \cite{DMS} or in the present paper.
   From the point of view of applications it is reasonable to relax
the assumptions on $A_t$ because it allows us to capture a larger class of financial models. With $A_{11}=A_{22}$ in the bivariate GARCH(1,1) we obtain 
$$\P (\s ^2_{1,t}>t)\sim (\log t)^{\a }t^{-\a },\quad  \P (\s
^2_{2,t}>t)\sim t^{-\a },$$
while the results of \cite{DMS} say 
$$\P (\s ^2_{1,t}>t)\sim t^{-\min (\a _1,\a _2) },\quad  \P (\s
^2_{2,t}>t)\sim t^{-\a _2}.$$


\section{Preliminaries and main results}
From now on $A_n$ in \eqref{rec} is a $2\times 2$ upper triangular matrix and $B_{n}\in \R ^2$. 
We assume that the entries on the diagonal of $A$ are equal and positive i.e. $A=[a_{ij}]$, with $a_{11}=a_{22}=a>0$, $a_{21}=0$, $a_{12}\in \R$. Let  $y=a^{-1}a_{12}$ and so $A$ is determined by the random variable $(y,a)\in \R \times \R ^+$. The vector $B$ will be written $(b_1,b_2)$.
Therefore, we have a sequence $(a_n,y_n,b_{1,n},b_{2,n})$ of i.i.d random variables such that
$$A_n= \left ( \begin{array}{cc}

 a_n&y_na_n\\

0&a_n 

\end{array}\right ), \quad B_n=(b_{1,n},b_{2,n})$$

Under very mild hypotheses the stationary solution $\Xa $ for \eqref{rec} exists and it is given by 
\begin{equation}\label{series}
\Xa =\sum _{n=1}^{\8 }A_1\cdots A_{n-1}B_n.
\end{equation}
Indeed, if $E\log ^+\| A\|<\8$ and $ E\log a<0$ then  the Lyapunov exponent
$$
\lim _{n\to \8}\frac{1}{n}\E \log \| A_1\cdots A_n\| <0$$
is strictly negative \cite{GGO} and so the series \eqref{series} converges a.s.\footnote{The statement in \cite{GGO} is much more general than what we need here and the proof is quite advanced. If 
 there is $\eps>0$ such that $\E  a^{\eps } <1$ and $\E (|y|^{\eps }+|b_1|^{\eps}+|b_2|^{\eps })<\8 $, then negativity of the Lapunov exponent follows quite easily, see \cite{S}, Proposition 7.4.5 and e.g \cite{DMS}. Finiteness of the above moments is assumed here anyway, see \eqref{ass2}, \eqref{ass3} and  \eqref{ass5}}.
We write
\begin{equation*}
A_n=a_nI+N_n,
\end{equation*}
where $N_n$ is an upper triangular matrix with zeros on the diagonal. Then $N_iN_j=0$ and so for $n\geq 2$
\begin{equation*}
A_1\cdots A_{n-1}=a_1\cdots a_{n-1} I +\sum _{j=1}^{n-1}a_1\cdots a_{j-1}N_ja_{j+1} \cdots a_{n-1}.
\end{equation*} 
Let $$ \Pi _n=a_1\cdots a_n$$ and
\begin{equation}\label{ikszero}
 \Xa _0=\sum _{n=2}^{\8}\Pi _{n-1}\big (\sum _{j=1}^{n-1}y_j\big )b_{2,n}.
 \end{equation}
Then $\Xa =(\Xa _1,\Xa _2)$, where
\begin{align}
\Xa _1=&\sum _{n=1}^{\8} \Pi _{n-1}b_{1,n}+ \Xa _0=\Xa '_1+\Xa _0\label{split1}\\
\Xa _2=&\sum _{n=1}^{\8} \Pi _{n-1}b_{2,n}
\end{align}
We are going to investigate the tail of $\Xa $. Our standing assumptions are:
\begin{equation}\label{ass1} 
\log a\  \mbox{is non-arithmetic}, \end{equation}
 there is $\a >0$ such that 
\begin{equation}\label{ass2}
\E a^{\a }=1\ \mbox{and} \ 0<\rho =\E a^{\a}\log a <\infty ,\end{equation} 
\begin{equation}\label{ass3}
\E (|b_1|^{\a } +|b_2|^{\a })<\8 \end{equation}
for every $x\in \R $
\begin{equation}\label{ass4}
\P (ax+b_2=x) <1 .\end{equation}
Under assumption $\E a^{\a }=1$, $\rho $ is strictly positive but it may be infinite. So finiteness is what we assume in \eqref{ass2}. Then the well known Kesten-Goldie Theorem (Theorem \ref{kestengoldie} in the Appendix), implies that
\begin{align}
\lim _{t\to \infty }\P [\Xa _2>t]t^{\a }&=c_+\label{righttail}\\
\lim _{t\to \infty }\P [\Xa _2<-t]t^{\a }&=c_- \label{lefttail}
\end{align}
$c_++c_->0$ but $c_+, c_-$ are not always strictly positive. However,  conditions for their strict positivity are easy to formulate, see \cite{BD} . Similarly, 
$$\P (|\Xa '_1|>t)\sim t^{-\a }$$
but in our case the tail of $\Xa _0$ is essentially heavier than that of $\Xa '_1$: the perturbation $\sum _{j=1}^{n-1}y_j$ in \eqref{ikszero} is responsible for the factor $(\log t)^{\a \slash 2}$ or $(\log t)^{\a }$ in \eqref{bothtail}. 
  
For the tail of $\Xa _0$ (or equivalently $\Xa _1$) we will need more assumptions
\begin{equation}\label{ass6} 
\log a\  \mbox{is non-lattice} \footnote{$\log a$ is not supported by the set of the form $c_1+c_2\Z$, where $\Z$ is the set of integers and $c_1,c_2\in \R$ are fixed.} \end{equation} 
and there is $\eps _0 >0$ such that 
\begin{equation}\label{ass5}
\E a^{\a +\eps _0}<\8 ,\ \E |b_2|^{\a +\eps _0}<\8 , \E \big (|y|^{\a +\eps _0}a^{\a +\eps _0}\big )<\8 .\end{equation}
More precisely,
\begin{equation*}
\P (\Xa _0>t)\sim \begin{cases} (\log t)^{\a }t^{-\a }, \quad \ \  \mbox{if}  \ \ \E ya^{\a }\neq 0\\ (\log t)^{\a \slash 2}t^{-\a }, \quad \mbox{if} \ \ \E ya^{\a }= 0.\end{cases}\end{equation*}
A short scheme of the proof is given below, preceded by exact formulations of our results.

\begin{thm}\label{mthm1}
Suppose that assumptions \eqref{ass2}-\eqref{ass4}, \eqref{ass6}, \eqref{ass5} are satisfied
and let $c_+,c_-$ be as in \eqref{righttail} and \eqref{lefttail}.
Assume further that $\E ya^{\a}=0$ and that there  is $r\geq 3$, $r>2\a +1$ such that $\E |y|^ra^{\a }<\8 $. Then 
\begin{equation}\label{zeroright}
\lim _{t \to \8}\P (\Xa _1>t)t^{\a }(\log t)^{-\a \slash 2 }=\lim _{t \to \8}\P (\Xa _1<-t)t^{\a }(\log t)^{-\a \slash 2 }=
(c_++c_-)\rho ^{\a \slash 2}c_0
\end{equation}
where $c_0=c_0(K)>0$ is defined in \eqref{wplus}.
\end{thm}

\begin{thm}\label{mthm}
Let $c_+,c_-$ be as in \eqref{righttail} and \eqref{lefttail}.
Suppose that assumptions \eqref{ass2}-\eqref{ass4}, \eqref{ass6}, \eqref{ass5} are satisfied. Assume further that $s=\E ya^{\a }\neq 0$ and there is $r\geq 3$, $r>\a $ such that $\E |y|^ra^{\a }<\8 $. Then 
\begin{equation}\label{right}
\lim _{t \to \8}\P (\Xa _1>t)t^{\a }(\log t)^{-\a }=\begin{cases} c_+s^{\a }\rho ^{\a }\ \ \ \mbox{if}\ s>0\\ c_-|s|^{\a }\rho ^{\a } \ \ \ \mbox{if}\ s<0\end{cases}
\end{equation}
and
\begin{equation}\label{left}
\lim _{t \to \8}\P (\Xa _1<-t)t^{\a }(\log t)^{-\a }=\begin{cases} c_-s^{\a }\rho ^{\a }\ \ \ \mbox{if}\ s>0\\ c_+|s|^{\a }\rho ^{\a } \ \ \ \mbox{if}\ s<0.\end{cases}
\end{equation}
\end{thm}
\begin{rem}
The simplest model of the case described in the previous theorem is obtained when $y=1$ and so the main term in \eqref{ikszero} becomes $\Xa _0(1)=\sum _{n=2}^{\8 }\Pi _{n-1}(n-1)b_{2,n}$. The idea is to establish asymptotics of $\Xa _0(1)$ and to compare $\Xa _0$ with $\Xa _0(1)$ by applying  Theorem \ref{mthm1} to 
$\Xa _0 - s \Xa (1)$.
\end{rem}
\begin{rem}
Clearly \eqref{zeroright}, \eqref{right} and \eqref{left} give relevant information only if $c_+$ or $c_-$ are strictly positive although the statement is true without that assumption. A simple necessary and sufficient condition for strict positivity of $c_+$ is given in Lemma 3.2 of \cite{BD}. Namely, let $\mu $ be the law of $(a,b_2)$. For $(u,v)\in \supp \mu $, $u\neq 1$ define $x(u,v)=\frac{v}{1-u}$. $c_+>0$ if and only if one of the following conditions is satisfied: $\P (a=1, b_2>0)>0$
or there are $(u_1,v_1), (u_2,v_2)$ in the support of $\mu $ such that $u_1>1$,  $u_2<1$ and $x(u_1,v_1)<x(u_2,v_2)$.
\end{rem}
\begin{proof}[Proof of Theorem \ref{mthm1} and \ref{mthm}-the scheme]
We split $\Xa _0$ into three parts
$$\Xa _0= N_t(L)+M_t(L)+N_{t,\8 }(L),$$ see \eqref{nleft} -\eqref{nright}.
Then, in view of \eqref{split1}
$$
\Xa _1= \Xa '_1+N_t(L)+M_t(L)+N_{t,\8 }(L).$$
In Section 3 we prove that $N_t(L), N_{t,\8 }(L)$ are negligible in the asymptotics, Lemma \ref{irrelevent}. Hence taking into account 
 Theorem \ref{kestengoldie} we obtain
$$
\P (\Xa '_1+N_t(L)+N_{t,\8 }(L)>t )=O(t^{-\a}).$$
So $M_t(L)$ is the main part and, if
$\E ya^{\a}=0$, $P(M_t(L)>t)$ is estimated in Section 4. \eqref{zeroright} follows from Lemmas \ref{mainzero} and \ref{offzero}. Section 5 is devoted to the case $\E ya^{\a}\neq 0$ and \eqref{right}, \eqref{left} follow from  Corollary \ref{nonzero}.
\end{proof}

Finally, we have the following ``degenerate'' regular behavior of $\Xa $.
\begin{cor}
Suppose that assumptions of Theorem \ref{mthm1} (or \ref{mthm} respectively) are satisfied.
Then for every $v=(v_1,v_2)\in \R ^2$
\begin{equation}\label{regular}
\lim _{t \to \8}\P [\langle v, \Xa \rangle >t]t^{\a }(\log t)^{-\tilde \a }=v_1
\lim _{t \to \8}\P [ \Xa _1  >t]t^{\a }(\log t)^{-\tilde \a }, 
\end{equation}
where $\tilde \a = \frac{\a }{2}$ (or $\tilde \a = \a $ respectively).  
\end{cor}

\section{Negligible parts of $\Xa _0$}
There are terms in \eqref{ikszero} that give the correct asymptotics and those that are irrelevent. We are going to discuss it now. 
Let
\begin{equation}\label{n0}
n_0=\rho ^{-1}\log t.
\end{equation}
and, given $D>0$,  let 
\begin{equation}\label{L}
L=L(t)=D\sqrt{(\log \log t) \log t}.\end{equation}
We split $\Xa _0$ into three parts
\begin{equation}\label{split}
\Xa _0= N_t(L)+M_t(L)+N_{t,\8 }(L),
\end{equation}
where
\begin{align}
N_t(L) &=\sum _{n=2}^{n_0-L-1}\Pi _{n-1}\big (\sum _{j=1}^{n-1}y_j\big )b_{2,n}\label{nleft}\\
M_t(L)&=\sum _{n=n_0-L}^{n_0+L}\Pi _{n-1}\big (\sum _{j=1}^{n-1}y_j\big )b_{2,n}\label{middle}\\
N_{t,\8}(L)&=\sum _{n=n_0+L+1}^{\infty }\Pi _{n-1}\big (\sum _{j=1}^{n-1}y_j\big )b_{2,n}\label{nright}
\end{align}
and we will prove that the terms \eqref{nleft} and \eqref{nright} are negligible. More precisely, we have the following lemma. 
\begin{lemma}\label{irrelevent}
Suppose that \eqref{ass2} and \eqref{ass5} are satisfied. For $\xi \geq 0$ there is $D_0$ such that for all $D\geq D_0$
\begin{align}
\P \big (N_t(L)>t\big )&=O\big (t^{-\a}(\log t)^{-\xi }\big )\label{smallind}\\
\P \big (N_{t,\8 }(L)>t\big )&=O\big (t^{-\a}(\log t)^{-\xi }\big )\label{largind}
\end{align}
\end{lemma}
\begin{proof}
We start with \eqref{smallind}. 
Since 
\begin{equation}\label{neg1}
\P \big (N_t(L)>t\big )\leq \sum _{n=2}^{n_0-L-1}\P \big (\Pi _{n-1}\big (\sum _{j=1}^{n-1}y_j\big )b_{2,n}>tn_0^{-1}\big ),
\end{equation}
we estimate just one term on the right hand side above.
By Chebychev inequality for $n\geq 2$ and $\eps \leq \eps _0$ we have
$$
\P \big (\Pi _{n-1}\big (\sum _{j=1}^{n-1}y_j\big )b_{2,n}>tn_0^{-1}\big )\leq
n^{\max (1,\a +\eps )}\E (|y|a)^{\a +\eps }\E |b_2|^{\a +\eps}\big (\E a^{\a +\eps}\big )^{n-2}t^{-(\a +\eps )}n_0^{\a +\eps }.
$$
Let $\g =\max (1,\a +\eps )+\a +\eps $. 
Writing $n=n_0-k$, we obtain
\begin{equation}\label{neg2}
\P \big (\Pi _{n-1}\big (\sum _{j=1}^{n-1}y_j\big )b_{2,n}>tn_0^{-1}\big )\leq t^{-(\a +\eps )}n_0^{\g}\sum _{k=L+1}^{n_0-2}\big (\E a^{\a +\eps})^{n_0-k}.
\end{equation}
Now we choose $\eps = \eps (k)\leq \eps _0$. Let $\Lambda (\b )=\log \E a ^{\b }$. Then, there is a constant $C_1$ such that 
$$
\E a^{\a +\eps}=e^{\Lambda (\a +\eps)}\leq e^{\Lambda '(\a )\eps +C_1\eps ^2}  
$$
for $0\leq \eps \leq \eps _0$ and so
\begin{align*}
\big (\E a^{\a +\eps}\big )^{n_0-k}&\leq  e^{(\rho \eps +C_1\eps ^2)(n_0-k)}\\
&= e^{\rho \eps n_0}e^{-k\rho \eps +C_1\eps ^2(n_0-k)}\\
&=t^{\eps } e^{-k\rho \eps +C_1\eps ^2n_0}.
\end{align*}
Taking $\eps =\frac{\rho k}{2C_1n_0}$, we obtain
$$
-k\rho \eps +C_1\eps ^2n_0\leq -\frac{\rho ^2 k^2}{4C_1n_0}\leq -\frac{\rho ^3D^2}{4C_1} \log \log t.
$$ 
Notice that taking $C_1$ possibly larger we can always guarantee that $\eps \leq \frac{\rho }{2C_1}<\eps _0$ in this calculation. If $D$ is large enough then
$\frac{\rho ^3D^2}{4C_1}>D>\g +2+\xi $ and finally, in view of \eqref{neg1} and \eqref{neg2}
$$
\P \big (N_t(L)>t\big )\leq \rho ^{-\g -2} (\log t)^{\g +2-D}t^{-\a }=O(t^{-\a}(\log )^{-\xi }).$$
For \eqref{largind}, writing $n=n_0+k$ and proceeding as before,  we have
\begin{align*}
\P \big (N_{t,\8 }(L)>t\big )&\leq 
\P \big (\Pi _{n-1}\big (\sum _{j=1}^{n-1}y_j\big )b_{2,n}>6\pi ^{-1}tk^{-2}\big )\\
&\leq
Cn^{\max (1,\a -\eps )}\E (|y|a)^{\a -\eps }\E |b_2|^{\a -\eps}\big (\E a^{\a -\eps}\big )^{n-2}t^{-(\a -\eps )}k^{2(\a -\eps )}.
\end{align*}
Moreover,
\begin{align*}
\big (\E a^{\a -\eps}\big )^{n_0+k-2}&\leq  e^{(-\rho \eps +C_1\eps ^2)(n_0+k-2)}\\
&\leq Ct^{-\eps } e^{-k\rho \eps +C_1\eps ^2(n_0+k)}.
\end{align*}
Now, taking $\eps =\frac{\rho k}{2C_1(n_0+k)}\leq \eps _0$, we obtain
$$
-k\rho \eps +C_1\eps ^2(n_0+k)= -\frac{\rho ^2 k^2}{4C_1(n_0+k)}.
$$ 
Therefore, 
$$
\P \big (N_{t,\8 }(L)>t\big )
\leq Ct^{-\a }\sum _{k=L+1}^{\8}(n_0+k)^{\max (1,\a -\eps )}k^{2(\a -\eps)}\exp\Big (-\frac{\rho ^2 k^2}{4C_1(n_0+k)}\Big ).
$$
For $k\leq n_0$ and $D\geq 8C_1\rho ^{-3}$, we have  
$$\frac{\rho ^2 k^2}{4C_1(n_0+k)}\geq \frac{\rho ^3D^2\log \log t}{8C_1}>D\log \log t$$ and if $k>n_0$ then 
 $$\frac{\rho ^2 k^2}{4C_1(n_0+k)}\geq \frac{\rho ^2k}{8C_1}\geq \frac{\rho ^2n_0}{8C_1}$$
Hence
$$
\P \big (N_{t,\8 }(L)>t\big )
\leq Ct^{-\a }\Big ((\log t)^{2+3(\a -\eps )-D}+\sum _{k=n_0+1}^{\8}k^{1+3(\a -\eps)}\exp\big (-\frac{\rho ^2 k}{8C_1}\big )\Big ).
$$
Finally, an elementary calculation shows that \eqref{largind} follows provided $D$ is large enough.
\end{proof}
In the same way we prove
\begin{lemma}\label{irrelev}
Assume that \eqref{ass2} is satisfied and $\E (a^{\a +\eps _0}+|b_1|^{\a +\eps _0}+|b_2|^{\a +\eps _0})<\8$.
Then for every $\b , \xi \geq 0$ there is $D$ such that for $j=1,2$
\begin{align}
\P \big (\sum _{n=1}^{n_0-L-1}\Pi _{n-1}|b_{j,n}|>t(\log t)^{-\b}\big )&=O\big (t^{-\a}\log t)^{-\xi}\big )\label{short}\\
\P \big (\sum _{n=n_0+L+1}^{\8 }\Pi _{n-1}|b_{j,n}|>t(\log t)^{-\b}\big )&=O\big (t^{-\a}(\log t)^{-\xi}\big )\label{largen}.
\end{align}
\end{lemma}
Lemma \ref{irrelev} will be used in the proof of Lemmas \ref{mainzero} and \ref{Ras} .
\section{The centered case}
In this section we assume that
$\E ya^{\alpha }=0$ and we study asymptotics of the main term 
\begin{equation*}
M_t=M_t(L)=\sum _{n=n_0-L}^{n_0+L}\Pi _{n-1}\big (\sum _{j=1}^{n-1}y_j\big )b_{2,n}
\end{equation*}
in $\Xa _0$. We are going to prove that
$$
\P (M_t >t)\sim t^{-\a }(\log t)^{\a \slash 2}$$
as $t\to \8$, see Lemmas \ref{mainzero} and \ref{offzero}. Let
$$
S_{2L}=\sum_{m=0}^{2L}a_{n_0-L}\cdots  a_{n_0-L+m-1}b_{2,n_0-L+m}
$$
$$
M'_t=a_1\cdots  a_{n_0-L-1}\big (\sum _{j=1}^{n_0-L-1}y_j\big )S_{2L}
$$
$$
M''_t=a_1\cdots  a_{n_0-L-1}\sum _{m=1}^{2L}\big (\sum _{j=n_0-L}^{n_0-L+m-1}y_j\big )a_{n_0-L+1}\cdots  a_{n_0-L+m-1}b_{2,n_0-L+m}
$$
Then
$$
M_t=M'_t+M''_t.$$
$M_t'$ is the main term in the asymptotics of $M_t$ and $M''_t$ is negligible.
For part of our calculations we are going to change the measure.
Namely, let $\F _n$ be the filtration defined by the sequence $(A_n, B_n)$ i.e. $\F _n=\s \big ( (A_j, B_j)_{j\leq n}\big )$. Then the expectation $\E _{\a }$ with respect to the new probability measure $\P _{\a}$ is defined by 
\begin{equation}\label{change}
\E _{\a }f =\E [f a_1^{\a}\cdots a_n^{\a}]\end{equation}
where $f$ is measurable with respect to $\F _n$. Notice that, in view of our assumptions, $\E _{\a }y_j=0$ and $\E _{\a }(\log a -\rho )=0$, $\E_{\a }(\log a)^2<\8$. Moreover, we assume that $\E_{\a }y^2$ is finite. This allows us to apply the central limit theorem to the sequence $(y_n, \log a_n-\rho )$. 
Let $K$ be the covariance matrix of $y$ and $\log a -\rho $ in the changed measure i.e.
$$K= \left ( \begin{array}{cc}

\E_{\a}y^2 &\E_{\a}y(\log a -\rho)\\

\E_{\a}y(\log a -\rho)&\E_{\a}(\log a -\rho)^2 

\end{array}\right ).$$
 We adopt the notation
\begin{equation}\label{wplus}
c_0(K)=\begin{cases}\E Z_1^{\a}\Ind \{Z_1\geq 0\},\ \mbox{where}\ (Z_1,Z_2)\ \mbox{has law}\ \mathcal{N} (0,K) \quad \mbox{if} \ \det K\neq 0, \\
\E Z_1^{\a}\Ind \{Z_1\geq 0\},\ \mbox{where}\ Z_1\ \mbox{has law}\ \mathcal{N} (0,1) \quad \mbox{if} \ \det K= 0.\end{cases}
\end{equation}
Now we are ready to formulate the main lemma.
\begin{lemma}\label{mainzero}
Suppose that assumptions \eqref{ass2}-\eqref{ass4}, \eqref{ass6}, \eqref{ass5} are satisfied
and let $c_+,c_-$ be as in \eqref{righttail} and \eqref{lefttail}.
Assume further that $\E ya^{\a}=0$ and that there  is $r\geq 3$, $r>\a $ such that $\E |y|^ra^{\a }<\8 $. Then 
\begin{equation*}
\lim _{t \to \8}\P ( M'_t>t )t^{\a }(\log t)^{-\a \slash 2 }=(c_++c_-)\rho ^{\a \slash 2}c_0(K).
\end{equation*}
\end{lemma}
\begin{proof}
We choose $L=L(t)$ such that \eqref{largen} is satisfied with $\xi =\a $. To simplify the notation, in this proof we will write 
$$n=n(t)=n_0-L-1=\frac{\log t}{\rho}-L(t)-1.$$ Notice that
in view of \eqref{L}
\begin{equation*}
\lim _{t\to \8}\frac{n}{\log t}=\rho 
\end{equation*}
and we will often write $n$ in place of $\log t$ in various expressions related to the asymptotics of $M'_t$.
Let
$$
\Y _n=\sum _{j=1}^ny_j. $$
Then 
$$
M'_t=\Pi _n \Y _nS_{2L}.
$$
\bigskip
 Since $\P (\Pi _n>t)\leq t^{-\a}$ it is enough to prove that
$$
\lim _{t\to \8}\P (M'_t>t, \Pi _n\leq t)t^{\a }(\log t)^{-\a \slash 2 }=(c_++c_-)\rho ^{\a \slash 2}c_0(K).$$
First we sketch the main steps of the proof. Then in Steps 1-3 below we do the detailed caculations. Finally, in Step 4 we conclude.

\medskip
{\bf Step 0. The outline of the proof.}
For a fixed $\dd >1$ (independent of $t$) and $d>(2(r-\a ))^{-1}$ we write
\begin{align*}
I_0(t,\dd )=&\P \left ( M'_t>t,\  \dd ^{-1}\sqrt{n}< |\Y _n|< \sqrt{n}(\log n)^d,\ \Pi _n\leq t\right )\\
I_1(t,\dd )=&\P \left ( M'_t>t,\  |\Y _n|\leq \dd ^{-1}\sqrt{n},\ \Pi _n\leq t\right )\\
I_2(t,\dd)=&\P \left ( M'_t>t,\ |\Y _n|\geq \sqrt{n}(\log n)^d,\  \Pi _n\leq t \right )
\end{align*}
and in Step 1 we prove that
\begin{align}
I_1(t,\dd )\leq &C t^{-\a }n^{\a \slash 2}\dd ^{-\a} \label{smallyy}\\
I_2(t,\dd)\leq &C t^{-\a }n^{\a \slash 2}(\log n)^{-d(r-\a)+1\slash 2} \label{largeyy}.
\end{align}
Above and in the rest of the proof all the constants do not depend on $t$ and $\dd$. 
\eqref{smallyy} and \eqref{largeyy} show that only $|\Y _n| $ ``close'' to $\sqrt{n}$ play the role.
Let
$$
S=\sum _{m=0}^{\8 }a_{n_0-L}\cdots a_{n_0-L+m-1}b_{2,n_0-L+m}
$$
Then $S$ is a perpetuity independent of $\Pi _n \Y _n$ and satisfying
$$
\lim _{n\to \8}\P (S>t)t^{-\a}=c_+.$$
Therefore, it is convenient to replace $S_{2L}$ by $S$ i.e.
 to compare the main term $I_0(t,\dd)$ with
$$
J((1\pm \eps )t,\dd )=\P \Big ( \Pi _n \Y _n S>(1\pm \eps )t,\ \dd ^{-1}\sqrt{n}< |\Y _n|< \sqrt{n}(\log n)^d,\ \Pi _n\leq t\Big )$$
Let 
$$
 H(\eps t,\dd )=\P \Big ( \Pi _n \Y _n|S-S_{2L}|>\eps t,\ \dd ^{-1}\sqrt{n}< |\Y _n|< \sqrt{n}(\log n)^d,\ \Pi _n\leq t\Big )$$ 
Then for every $\eps >0$
$$
J((1+\eps )t,\dd )-H(\eps t,\dd)\leq I_0( t,\dd)\leq J((1-\eps )t,\dd )+H(\eps t,\dd)$$  
Notice that $\Pi_n|S-S_{2L}|\leq \sum _{m=L}^{\8}\Pi_{n_0+m}|b_{2,n_0+m+1}|$. Hence
by \eqref{largen} and our choice of $L(t)$, we have
\begin{equation}\label{H}
H(\eps t)\leq \P \Big ( \Pi _n |S- S_{2L}|>\eps t n ^{-\jd}(\log n)^{-d}\Big )\leq C\eps ^{-\a }t^{-\a}. 
\end{equation}
Then by \eqref{smallyy}, \eqref{largeyy} for $\xi =d(r-\a)-1\slash 2>0$ and \eqref{H}, we have
\begin{align*}
&J((1+\eps )t,\dd )t^{\a}n^{-\a \slash 2}-C\eps ^{-\a }n^{-\a \slash 2}\leq \P (M'_t>t, \Pi _n\leq t)t^{\a}n^{-\a \slash 2}\\
&\leq J((1-\eps )t,\dd )t^{\a}n^{-\a \slash 2}+C\dd ^{-\a}+C(\log n)^{-\xi}+C\eps ^{-\a }n^{-\a \slash 2}.
\end{align*}
Suppose now we can prove that
\begin{equation}\label{J}
\lim _{\dd \to \8}\lim _{t\to \8}J((1\pm \eps )t,\dd )t^{\a}n^{-\a \slash 2}=(c_++c_-)c_0(K)(1\pm \eps)^{-\a}. \end{equation}
Then Lemma \ref{mainzero} follows.
The crucial quantity in getting \eqref{J} is 
\begin{equation}\label{defI}
I(n,\dd )=\E \Pi _n^{\a}\big (n^{-\jd}\Y _n\big )^{\a }\Ind {\{ \dd ^{-1}\sqrt{n}< \Y _n < \dd \sqrt{n}\} }\Ind \{\Pi _n\leq t \}.
\end{equation}
In Step 2 we prove that 
\begin{equation}\label{central}
\lim _{\dd \to \8}\lim _{t\to \8}I(n,\dd)=c_0(K).
\end{equation}
Then we estimate the error
\begin{equation}\label{error}
|J((1\pm \eps )t,\dd )t^{\a}n^{-\a \slash 2}-(c_++c_-)I(n,\dd)(1\pm \eps )^{-\a}|
\end{equation}
well enough, see \eqref{estimate} in Step 3. Finally, in Step 4, we conclude \eqref{J}.

\medskip
{\bf Step 1. Proof of \eqref{smallyy}, \eqref{largeyy}.}\
Fix $\dd >1$ and suppose that $|\Y _n|\leq \sqrt{n}\dd ^{-1}$. In view of \eqref{righttail}, \eqref{lefttail} we have
\begin{equation}\label{smally}
\P (|\Y _n|\leq \sqrt{n}\dd ^{-1}, \Pi _{n}|\Y _nS_{2L}|>t )
\leq \P (\Pi _{n}|S_{2L}|>tn^{-1\slash 2}\dd  )
\leq C \ t^{-\a }n^{\a \slash 2}\dd ^{-\a} ,  
\end{equation}
which can be made arbitrarily small provided $\dd $ is large enough.
Now we consider large $\Y _n$. For a fixed $d>(2(r-\a))^{-1}$, we define the sets
\begin{align*}
W_k&=\{ e^{k-1} \sqrt{n} < |\Y _n|\leq e^k \sqrt{n}\}, \ k\geq k_0=d\log \log n\\
Z_0&=\{ |S_{2L}|\leq 1\}\\
Z_p&=\{ e^{p-1} < |S_{2L}|\leq e^p \}, \ p\geq 1.\\
\end{align*}
and we estimate 
$$
\P (\Pi _n |\Y _n||S_{2L}|>t, |\Y _n|> (\log n)^{k_0-1}\sqrt{n} )=
\sum _{k\geq k_0, p\geq 0}\P \big (\{ \Pi _n |\Y _n||S_{2L}|>t\} \cap W_k \cap Z_p\big ).
$$
We are going to change the measure (see \eqref{change}) and to prove that
\begin{equation}\label{largey}
\P (\Pi _n |\Y _n||S_{2L}|>t, |\Y _n|> (\log n)^{k_0-1}\sqrt{n} )=o\big (t^{-\a }(\log t)^{\a \slash 2}\big ).
\end{equation}
For fixed $k,p$ we have
\begin{align*}
\P \Big (\{ \Pi _n |\Y _n||S_{2L}|>t\} \cap W_k \cap Z_p\Big )&\leq
\P \Big (\{ \Pi _n|\Y _n|>te^{-p}\} \cap W_k \Big )\P (Z_p)\\
&\leq \P \Big (\{ |\Y _n|> e^{k-1}\sqrt{n}, \Pi _n>te^{-k-p}n^{-1\slash 2} \}  \Big )\P (Z_p)\\
&\leq \P (Z_p) \E _{\a }\Ind \{|\Y _n|> e^{k-1}\sqrt{n}\}\Ind \{\Pi _n>te^{-k-p}n^{-1\slash 2} \}\Pi _n ^{-\a}\\
&\leq t^{-\a }n^{\a \slash 2}e^{(k+p)\a}\P _{\a } \big ( |\Y _n|>e^{k-1}\sqrt{n}\big ) \P (Z_p).
\end{align*} 
To estimate the last term we use Edgeworth expansions: Theorem \ref{petrov1}. The latter says that there is a constant $C_1=C_1(r,\E _{\a}y^2, \E _{\a}|y|^3,\E _{\a}|y|^r)$ such that
$$
\P _{\a } \big (|\Y _n|>e^{k-1}\sqrt{n}\big )\leq C_1 n^{-\jd} e^{-r(k-1)}.
$$
Moreover, if $p\geq 3\rho L$ then by \eqref{short} there is $C_2$ such that
$$
\P (S_{2L}>e^{p-1})\leq C_2 e^{-\a p}p^{-2}.$$
Hence
\begin{align*}
\sum _{p\geq 3\rho L,k\geq k_0} \P \Big (\{ \Pi _n|\Y _n||S_{2L}|>t\} \cap W_k\cap Z_p \Big )&\leq C_1C_2\sum _{p\geq 3\rho L,k\geq k_0}t^{-\a }n^{\a\slash 2 }e^{(k+p)\a} n^{-\jd} e^{-r(k-1)}e^{-\a p}p^{-2}\\
&\leq C_3 t^{-\a }n^{\a \slash 2} n^{-\jd}.
\end{align*}
If $p< 3\rho L$ then, by \eqref{righttail} $\P (S_{2L} >e^{p-1}) \leq C_4e^{-\a p}$ and so 
\begin{align*}
\sum _{p< 3\rho L,k\geq k_0} \P \Big (\{ \Pi _n|\Y _n||S_{2L}|>t\} \cap W_k\cap Z_p \Big )&\leq C_1C_4L\sum _{k\geq k_0}n^{-\jd}e^{-r(k-1)}t^{-\a }n^{\a\slash 2 } e^{\a k} \\
&\leq C_5 t^{-\a }n^{\a \slash 2}(\log n)^{-d(r-\a )+\jd} .
\end{align*}
Hence \eqref{largey} follows.

\bigskip
{\bf Step 2. Proof of \eqref{central}.}\
Let $S_n=\sum _{j=1}^n(\log a_j -\rho )$. 
We are going to apply the central limit theorem to the sequence 
$$
\Big (\frac{\Y _n}{\sqrt{n}}, \frac{S_n}{\sqrt{n}}\Big ).$$ 
Since $\log t=\rho n_0=\rho n +\rho (n_0-n)$, we have
\begin{equation*}
I(n,\dd )=\E _{\a }\big (n^{-1\slash 2}\Y _n\big )^{\a }\Ind {\{ \dd ^{-1}\sqrt{n}< \Y _n < \dd \sqrt{n}\} }\Ind \left \{\frac{S_n}{\sqrt{n}}\leq \frac{\rho(n_0-n)}{\sqrt{n}}\right \}.
\end{equation*}
Notice that $\frac{\rho(n_0-n)}{\sqrt{n}}\to \8$ when $n\to \infty $. For
the covariance matrix $K$ of the variables $y_1$ and $\tilde X_1$, we distinguish two cases: $\det K\neq 0$ and $\det K=0$. 

Suppose first that $\det K\neq 0$. Let $\theta >1$ but close to $1$ and let $M>0$ be large. Define
$$
j_0=\min \{ j: \theta ^j\geq \dd ^{-1}\}\quad \mbox{and}\quad  j_1=\max \{ j: \theta ^j\leq \dd\}.
$$
Then
$$
I(n,\dd )\leq \sum _{j=j_0-1}^{j_1 }\theta ^{(j+1)\a }\E _{\a }\Ind \{\theta ^j\leq \frac{\Y _n}{\sqrt{n}}\leq \theta ^{j+1}\}=U(n,\theta, \dd).
$$
and, for sufficiently large $n$, 
$$
I(n,\dd )\geq \sum _{j=j_0}^{j_1-1 }\theta ^{j\a }\E _{\a }\Ind \{\theta ^j\leq \frac{\Y _n}{\sqrt{n}}\leq \theta ^{j+1}\}\Ind \{\frac{S_n}{\sqrt n}\leq M\}=L(n,\theta ,\dd).
$$
For a fixed $j$, we have
$$
\E _{\a }\Ind \{\theta ^j\leq \frac{\Y _n}{\sqrt{n}}\leq \theta ^{j+1}\}\to \frac{1}{2\pi \sqrt{\det K}}\int _{\theta ^j}^{ \theta ^{j+1}}\int _{\R }\exp \big (-\frac{1}{2}\langle z,K^{-1}z\rangle \big )\ dz, \quad \mbox{as}\ n\to \infty ,
$$
where $dz=dz_1dz_2$, and
$$
\E _{\a }\Ind \{\theta ^j\leq \frac{\Y _n}{\sqrt{n}}\leq \theta ^{j+1}\}\Ind \{\frac{S_n}{\sqrt n}\leq M\}\to \frac{1}{2\pi \sqrt{\det K}}\int _{\theta ^j}^{ \theta ^{j+1}}\int _{-\infty }^M\exp \big (-\frac{1}{2}\langle z,K^{-1}z\rangle \big )\ dz, \quad \mbox{as}\ n\to \infty.
$$
Hence
$$
\lim _{n\to \infty }U(n,\theta, \dd )=\sum _{j=j_0-1}^{j_1}
\frac{\theta ^{(j+1)\a }}{2\pi \sqrt{\det K}}\int _{\theta ^j}^{ \theta ^{j+1}}\int _{\R }\exp \big (-\frac{1}{2}\langle z,K^{-1}z\rangle \big )\ dz
$$
and so 
\begin{align*}
0\leq \lim _{n\to \infty }U(n,\theta, \dd )&-\frac{1}{2\pi \sqrt{\det K}}\int _{\theta ^{j_0-1}}^{ \theta ^{j_1+1}}\int _{\R }z_1^{\a}\exp \big (-\frac{1}{2}\langle z,K^{-1}z\rangle \big )\ dz\\
&\leq (\theta ^{\a }-1)
\frac{1}{2\pi \sqrt{\det K}}\int _{\theta ^{-1}\dd ^{-1}}^{ \theta \dd}\int _{\R }z_1^{\a}\exp \big (-\frac{1}{2}\langle z,K^{-1}z\rangle \big )\ dz\\
&\leq (\theta ^{\a }-1)
\frac{1}{2\pi \sqrt{\det K}}\int _0^{\infty }\int _{\R }z_1^{\a}\exp \big (-\frac{1}{2}\langle z,K^{-1}z\rangle \big )\ dz.
\end{align*}
In the same way we prove that
$$
\lim _{n\to \infty }L(n,\theta, \dd )=\sum _{j=j_0}^{j_1-1}
\frac{\theta ^{j\a }}{2\pi \sqrt{\det K}}\int _{\theta ^j}^{ \theta ^{j+1}}\int _{-\8 }^M\exp \big (-\frac{1}{2}\langle z,K^{-1}z\rangle \big )\ dz
$$
and
\begin{align*}
0\leq \frac{1}{2\pi \sqrt{\det K}}\int _{\theta ^{j_0}}^{ \theta ^{j_1}}&\int _{-\8 }^Mz_1^{\a}\exp \big (-\frac{1}{2}\langle z,K^{-1}z\rangle \big )\ dz-\lim _{n\to \infty }L(n,\theta, \dd )\\
&\leq (\theta ^{\a }-1)
\frac{1}{2\pi \sqrt{\det K}}\int _0^{\infty }\int _{\R }z_1^{\a}\exp \big (-\frac{1}{2}\langle z,K^{-1}z\rangle \big )\ dz.
\end{align*}
Therefore, for every $\d >0$ there is $\theta _0$ such that for $1<\theta \leq \theta _0$
\begin{align*}
-\d +\frac{1}{2\pi \sqrt{\det K}}\int _{\theta \dd ^{-1}}^{\theta ^{-1}\dd }\int _{-\8}^M z_1^{\a}&\exp \big (-\frac{1}{2}\langle z,K^{-1}z\rangle \big )\ dz\leq \liminf _{n\to \8}I(n,\dd)\\
\leq \limsup _{n\to \8}I(n,\dd)\leq \d +&
\frac{1}{2\pi \sqrt{\det K}}\int _{\theta ^{-1}\dd ^{-1}}^{\theta \dd }\int _{\R }z_1^{\a}\exp \big (-\frac{1}{2}\langle z,K^{-1}z\rangle \big )\ dz.
\end{align*}
Now, letting $\theta \to 1$, $M\to \8 $ and $\d \to 0$ we obtain
that 
\begin{equation}\label{constD}
\lim _{n\to \infty} I(n,\dd )=\frac{1}{2\pi \sqrt{\det K}}\int _{\dd ^{-1}}^{\dd}z_1^{\a }\int _{\R}\exp \big (-\frac{1}{2}\langle z,K^{-1}z\rangle \big )\ dz =c_0(\dd ,K) .
\end{equation}
Finally, letting $\dd \to \infty $, we obtain \eqref{central}. 

If $\det K=0$ then $y=\lambda (\log a -\rho )$, $\Ind \{\Pi _n\leq t\}$ in \eqref{defI} doesn't bring any restriction and so 
$$
I(n,\dd )=\E _{\a }\big (n^{-1\slash 2}\Y _n\big ) ^{\a }\Ind {\{\dd ^{-1}\leq \frac{\Y _n}{\sqrt{n}}\leq \dd \}}.
$$ 
As before, 
\begin{equation}\label{constDD}
I(n,\dd ) \to  (2\pi)^{-1\slash 2}\int _{\dd ^{-1}}^{\dd }z_1^{\a }\exp \big (-\frac{1}{2}z_1^2\big )\ dz_1=c_0(\dd).
\end{equation}
Finally, letting $\dd \to \infty $ 
we obtain \eqref{central}. Further on we will use notation $c_0(\dd,K)$ in both cases.

\bigskip
{\bf Step 3. Estimate of \eqref{error}.}
Let
$$
J^+((1+\eps )t,\dd )=\P \Big ( \Pi _n \Y _n S>(1+\eps )t,\ \dd ^{-1}\sqrt{n}< \Y _n< \sqrt{n}(\log n)^d,\ \Pi _n\leq t\Big )$$
$$
J^-((1+\eps )t,\dd )=\P \Big ( \Pi _n \Y _n S>(1+\eps )t,\ - \sqrt{n}(\log n)^d<\Y _n<-\dd ^{-1}\sqrt{n},\ \Pi _n\leq t\Big )$$
Then $J((1+\eps )t,\dd )=J^+((1+\eps )t,\dd )+J^-((1+\eps )t,\dd ).$
We have
$$
J^+((1+\eps)t,\dd)=\int _{\R ^+\times \R }\Ind _{U_n}\P \left (S>(ay)^{-1}(1+\eps )t\right )\Ind \{a \leq t\}\ d\mu _n(a,y),
$$
where $U_n=\{ \dd ^{-1}\sqrt{n}\leq y \leq \sqrt{n}(\log n)^d\}$ and $\mu _n$ is the law of $(\Pi _n , \Y _n)$. 

\medskip
First we compare $t^{\a}n^{-\a \slash 2}J^+((1+\eps)t,\dd )$ with $c_+\mathcal{L}(t,\dd)(1+\eps)^{-\a}$, where
\begin{equation*}
\mathcal{L}(t,\dd)= \E \Pi _n^{\a}(n^{-\jd}\Y _n\big )^{\a }\Ind _{U_n}\Ind \{\Pi _n\leq t \}.
\end{equation*}
Given $\eta >0$, we choose $T$ such that 
$$
|\P (S>s)s^{\a}-c_+|< \eta$$
for $s>T$.
Let
$$
P (a,y,t(1+\eps ))= \P (S>(ay)^{-1}(1+\eps) t)t^{\a}(1+\eps)^{\a}(ay)^{-\a}.$$
Notice that if $t(ya)^{-1}>T$ then $|P (a,y,t(1+\eps ))-c_+|<\eta $. We have
\begin{align*}
t^{\a }n^{-\a \slash 2}J^+((1+\eps)t,\dd)=&(1+\eps)^{-\a}\int _{\R ^+\times \R }{\bf 1}_{U_n}P (a,y,t(1+\eps ))\\
&n^{-\a \slash 2} \Ind \{a \leq t\} \Ind \{ya < tT^{-1}\}\ (ay)^{\a}\ d\mu _n(a,y)\\
+&(1+\eps)^{-\a}\int _{\R ^+\times \R }{\bf 1}_{U_n}P (a,y,t(1+\eps ))\\
&n^{-\a \slash 2} \Ind \{ a \leq t\} \Ind \{ ya \geq tT^{-1}\}\ (ay)^{\a}\ d\mu _n(a,y)
\end{align*}
and we decompose $\mathcal{L}(t,\dd)$ accordingly. More precisely, 
\begin{equation*}
\mathcal{L}(t,\dd)=\mathcal{L} _1(t,\dd)+\mathcal{L} _2(t,\dd),
\end{equation*}
where
\begin{align*}
\mathcal{L} _1(t,\dd):=& \E \Pi _n^{\a}(n^{-\jd}\Y _n\big )^{\a }\Ind _{U_n}\Ind \{\Pi _n\leq t \}\Ind \{\Pi _n\Y _n<t T^{-1}\}\\
=& \int _{\R ^+\times \R }{\bf 1}_{U_n}n^{-\a \slash 2}\Ind \{a \leq t\}\Ind \{ya <tT^{-1}\}\ (ay)^{\a}\ d\mu _n(a,y)
\\
\mathcal{L} _2(t,\dd):=&\E \Pi _n^{\a}(n^{-\jd}\Y _n\big )^{\a }\Ind _{U_n}\Ind \{\Pi _n\leq t \}\Ind \{\Pi _n\Y _n\geq t T^{-1}\}\\
=& \int _{\R ^+\times \R }{\bf 1}_{U_n}n^{-\a \slash 2}\Ind \{a \leq t\}\Ind \{ya \geq tT^{-1}\}\ (ay)^{\a}\ d\mu _n(a,y)
\end{align*}
and so
\begin{multline}
|t^{\a }n^{-\a \slash 2}J^+((1+\eps)t,\dd)-c_+\mathcal{L}(t,\dd)(1+\eps)^{-\a}|\\
\leq \eta \mathcal{L} _1(t,\dd)(1+\eps)^{-\a}
+C(1+\eps)^{-\a}\mathcal{L} _2(t,\dd)\\
\leq \eta \mathcal{L} _1(t,\dd)(1+\eps)^{-\a}+C(\log n +\log (T\dd))n^{-1\slash 4}.
\label{difference}
\end{multline}
Let $V_n=\{tT^{-1}n^{-1\slash 2}(\log n)^{-d}\leq \Pi _n\leq t \}$. To prove the last inequality, in \eqref{difference} we write 
\begin{align*}
\mathcal{L} _2(t,\dd)
&\leq \E \Pi _n^{\a}(n^{-\jd}\Y _n\big )^{\a }\Ind _{U_n}\Ind _{V_n}\\
&\leq  (\log n)^{\a d}  \E \Pi _n^{\a}\Ind _{V_n}\leq C (\log n+\log (T))n^{-1\slash 4}.
\end{align*}
Indeed, by \ref{petrov2}, 
\begin{align*}
\E \Pi _n^{\a}\Ind _{V_n}&\leq 
\sum _{m=0}^{\log n +\log (T)}\E \Pi _n^{\a}\Ind \{te^{-m-1}\leq \Pi _n \leq te^{-m}\}\\
&\leq \sum _{m=0}^{\log n +\log (T) }t^{\a }e^{-\a m}\P (\Pi _n\geq te^{-m-1} )
\leq C (\log n +\log (T))n^{-1\slash 2}.
\end{align*}
Now it remains to replace $\mathcal{L} (t,\dd)$ by $I(n,\dd)$. We have
\begin{multline}\label{LI}
\mathcal{L}(t,\dd)-I(n,\dd )=
\E \Pi _n^{\a}\big (n^{-\jd}\Y _n\big )^{\a }\Ind {\{ \dd \sqrt{n}\leq \Y _n \leq \sqrt{n}(\log n)^d\} }\\
\leq C\frac{r}{r-\a}D^{-(r-\a)}.
\end{multline}
For \eqref{LI} we write 
$$
\E \Pi _n^{\a}\big (n^{-\jd}\Y _n\big )^{\a }\Ind {\{ \dd \sqrt{n}\leq \Y _n \leq \sqrt{n}(\log n)^d\} }\leq \int _{\dd}^{\infty }s^{\a}\ dF_n(s),
$$
where $F_n$ is the distribution function of $\frac{\Y _n}{\sqrt{n}}$ with respect to the changed measure. Let $\bar F_n=1-F_n$. Then by Theorem \ref{petrov1} for $s\geq D$,
\begin{equation*}
\bar F_n(s)\leq C s^{-r}.
\end{equation*}
Hence
\begin{align*}
\int _{\dd}^{\8 }s^{\a}\ d F_n(s)&=-\bar F_n(s)s^{\a} \big |_{\dd}^{\8}+\a  \int _{\dd}^{\infty}s^{\a -1} \bar F_n(s)\ ds\\
&\leq C{\dd} ^{-(r-\a )} + \a C \int _{\dd }^{\8} s^{\a - r -1}\ ds\\
&\leq  C\frac{r}{r-\a}{\dd} ^{-(r-\a )}
\end{align*}
and \eqref{LI} follows. Therefore, in view of \eqref{difference} and \eqref{LI}, for every $\eta $
\begin{align*}
|t^{\a }n^{-\a \slash 2}J^+((1+\eps)t,\dd)&-c_+I(n,\dd)(1+\eps)^{-\a}|\leq \eta I(n,\dd)(1+\eps)^{-\a}\\
+C(\log n&+\log (T))n^{-1\slash 4}+C(1+\eps)^{-\a}{\dd} ^{-(r-\a )} .
\end{align*}
For $$|t^{\a }n^{-\a \slash 2}J^-((1+\eps)t,\dd)-c_-I(n,\dd)(1+\eps)^{-\a}|$$
we obtain the same bound. Hence
\begin{multline}\label{estimate}
|t^{\a }n^{-\a \slash 2}J((1+\eps)t,\dd)-(c_++c_-)I(n,\dd)(1+\eps)^{-\a}|\leq 2\eta I(n,\dd)(1+\eps)^{-\a}\\
+C(\log n+\log (T))n^{-1\slash 4}+C(1+\eps)^{-\a}{\dd} ^{-(r-\a )} .
\end{multline}

\medskip
{\bf Step 4. Conclusion.}\
In view of \eqref{constD} and \eqref{constDD}
$$
\lim _{t\to \8}I(n,\dd )=c_0(\dd , K)\leq c_0(K)$$
Hence letting first $t\to \8$ then $\eta \to 0$ in \eqref{estimate} we have
$$
\limsup _{t\to \8}t^{\a }n^{-\a \slash 2}J((1+\eps)t,\dd)\leq (c_++c_-)c_0(\dd ,K)(1+\eps )^{-\a}+C(1+\eps)^{-\a}{\dd} ^{-(r-\a )}.$$
and
$$
\liminf _{t\to \8}t^{\a }n^{-\a \slash 2}J((1+\eps)t,\dd)\geq (c_++c_-)c_0(\dd ,K)(1+\eps )^{-\a}-C(1+\eps)^{-\a}{\dd} ^{-(r-\a )}$$
But
$$\lim _{\dd \to \8}c_0(\dd ,K)=c_0(K).$$
Hence
$$
\lim _{\dd\to \8}\lim _{t\to \8}t^{\a }n^{-\a \slash 2}J((1+\eps)t,\dd)=(c_++c_-)c_0(K)(1+\eps )^{-\a}.$$
In the same way we prove
$$
\lim _{\dd\to \8}\lim _{t\to \8}t^{\a }n^{-\a \slash 2}J((1-\eps)t,\dd)=(c_++c_-)c_0(K)(1-\eps )^{-\a}$$
and \eqref{J} follows. \end{proof}
Now using Edgeworth expansions (Theorem \ref{petrov1}) we can estimate $M''_t$.  
\begin{lemma}\label{offzero}
Suppose that assumptions \eqref{ass1}-\eqref{ass4}, \eqref{ass5} are satisfied.
Assume further that $\E ya^{\a }=0$ and that there is $r>2\a +1$,  $r\geq 3$, such that $\E |y|^ra^{\a }<\infty $. Then there are $C>0$ and $\b <\frac{\a}{2} $ such that
\begin{equation}\label{off}
\P ( | M''_t|>t )\leq Cn^{\b }t^{-\a }.
\end{equation}
If $r>\a $ but not necessarily $r> 2\a +1$ then we have \eqref{off} with $\b <\a$. 
\end{lemma}
\begin{proof}
Let  $\Y _m=\sum _{j=n_0-L}^{n_0-L+m-1}y_j$, $1\leq m\leq 2L$. To simplify the notation we will write $n=n_0-L-1$. Let $0<\s <\frac{r-2\a -1}{2r}$. Recall that in this notation
$$
M''_t=\sum _{m=1}^{2L} \Y _m\Pi _{n+m}b_{2,n+m+1}.$$
We have
\begin{align*}
\P ( |M''_t|>t )&\leq \P \Big ( \sum _{m=1}^{2L} \big (|\Y _m |\Ind {|\Y _m |>n^{\frac{1-\s }{2}}}+n^{\frac{1-\s }{2}}\big )\Pi _{n+m}|b_{2,n+m+1}|>t\Big )\\
&\leq \P \Big ( \sum _{m=1}^{2L} \Pi _{n+m}|b_{n+m+1}| >\frac{t}{2}n^{-\frac{1-\s }{2}}\Big )\\
&+ \P \Big ( \sum _{m=1}^{2L} \big (|\Y _m |\Ind {|\Y _m |>n^{\frac{1-\s }{2}}}\big )\Pi _{n+m}|b_{n+m+1}|>\frac{t}{2}\Big )\\
&\leq Ct^{-\a }n^{\frac{\a (1-\s )}{2}} +
\sum _{m=1}^{2L}\P \Big ( |\Y _m |\Ind {|\Y _m |>n^{\frac{1-\s }{2}}} \Pi _{n+m} |b_{n+m+1}|>\frac{ct}{2m^{1+\s}}\Big ),
\end{align*}
where $c^{-1}=\sum _{m=1}^{\8}\frac{1}{m^{1+\s}}$. Fix $m$ and for $k\geq 1$ define 
\begin{equation*}
W_{k}=\{ e^{k-1} n^{\frac{1-\s}{2}}< |\Y _m|\leq e^k n^{\frac{1-\s}{2}}\}, \ k\geq 1,\\
\end{equation*}
and
$$U_{t,k,s}=\{ \Pi _{n+m}>\frac{c}{2}ts^{-1}e^{-k}n^{-\frac{(1-\s)}{2}}m^{-(1+\s)}\}.$$
Let $\nu $ be the law of $|b_2|$.
Then 
\begin{align*}
&\P \Big ( |\Y _m |\Ind {\{|\Y _m |>n^{\frac{1-\s }{2}}\}} \Pi _{n+m} |b_{n+m+1}|>\frac{ct}{2m^{1+\s}}\Big )\\&=
\int _{\R }\P \Big ( |\Y _m |\Ind {\{|\Y _m |>n^{\frac{1-\s }{2}}\}} \Pi _{n+m} >\frac{ct}{2sm^{1+\s}}\Big )\Ind {s\neq 0}\ d\nu (s)\\
&\leq \sum _{k=1}^{\8 }\int _{\R }\P \Big ( \big \{ |\Y _m |\Ind {\{|\Y _m |>n^{\frac{1-\s }{2}}\}} \Pi _{n+m}>\frac{ct}{2sm^{1+\s}}\big \}\cap W_k\Big )\Ind {s\neq 0}\ d\nu (s)\\
&\leq \sum _{k=1}^{\8 }\int _{\R}\E \Ind _{W_k}\Ind _{U_{t,k,s}}\ d\nu (s)=
 \sum _{k=1}^{\8 }\int _{\R } \E _{\a }\Ind _{W_k}\Ind _{U_{t,k,p}}\Pi _{n+m}^{-\a }\big )\ d\nu (s)\\
&\leq Ct^{-\a }m^{\a (1+\s)} n^{\frac{(1-\s)\a}{2}}\sum _{k=1}^{\8 }e^{k\a}\P _{\a } \big (|\Y _m|>e^{k-1}n^{\frac{(1-\s)}{2}}\big )\int _{\R }s^{\a }\ d\nu (s).
\end{align*} 
Observe that $m^{-\jd} n^{\frac{(1-\s)}{2} }\geq L^{-\jd }n^{\frac{1-\s}{2}}\to \8 $ when $n\to \8 $. Hence, as before,  by Theorem \ref{petrov1} there is $C_1$ 
$$
\P _{\a } \big (m^{-\jd}|\Y _m|>e^{k-1}n^{\frac{(1-\s)}{2}}m^{-\jd}\big )\leq C_1 m^{\frac{r}{2}-\jd} e^{-r(k-1)}n^{-\frac{r(1-\s)}{2} }
$$
for sufficiently large $n$ and all $m$. Summing over $k$, we obtain 
$$
\P \Big ( |\Y _m |\Ind _{\{|\Y _m |>n^{\frac{1-\s }{2}}\}}\Pi _{n+m} |b_{n+m+1}|>\frac{ct}{2m^{1+\s}}\Big )\leq C_2
t^{-\a }n^{\frac{(1-\s)\a}{2}} m^{\frac{r}{2}-\jd +(1+\s )\a} n^{-\frac{r(1-\s)}{2} }\E |b_2|^{\a }
$$
Finally taking the sum over $m$ we get 
\begin{align*}
\P \Big ( \sum _{m=1}^{2L} \big (|\Y _m |\Ind _{\{|\Y _m |>n^{\frac{1-\s }{2}}\}}\big )\Pi _{n+m}|b_{n+m+1}|>\frac{t}{2}\Big )
&\leq C_3t^{-\a }L^{\frac{r}{2}+\jd +(1+\s )\a}n^{\frac{(1-\s)\a}{2}-\frac{r(1-\s)}{2} }\\
& \leq C_3t^{-\a }n^{\b }(\log n)^{\gamma },
\end{align*}
where
$$
\b =\frac{-r+1+2r\s +4\a}{4}<\frac{\a }{2}, \quad \gamma =\frac{r+1}{4}+ \frac{(1+\s)\a}{2} $$
because $L\leq D\sqrt{n\log n}$ and $\s <\frac{r-2\a -1}{2r}$. Notice that if $r\geq 3$ but not necessarily $r>2\a +1$ and $\s $ is small enough then $\b <\a$.
\end{proof}


\section{The non centered case}
Now we assume that
$\E ya^{\alpha }=s\neq 0$ and we study asymptotics of the main term $M_t$ in $\Xa _0$.
Let
\begin{equation*}
R_t=\sum _{m=n_0-L}^{n_0+L}\Pi _{m-1}(m-1)b_{2,m}.
\end{equation*}
Then
\begin{equation*}
M_t-sR_t=\sum _{m=n_0-L}^{n_0+L}\Pi _{m-1}\big (\sum _{j=1}^{m-1}(y_j-s)\big )b_{2,m}
\end{equation*}
and in view of Lemma \ref{mainzero} and the second statement of Lemma \ref{offzero} we have 
\begin{lemma}
Suppose that assumptions \eqref{ass2}-\eqref{ass4}, \eqref{ass6}, \eqref{ass5} are satisfied
and let $c_+,c_-$ be as in \eqref{righttail} and \eqref{lefttail}.
Assume further that there  is $r\geq 3$, $r>\a $ such that $\E |y|^ra^{\a }<\8 $. Then  
$$
\lim _{t \to \8}\P ( |M_t-sR_t|>t )t^{\a }(\log t)^{-\a  }=0.
$$
\end{lemma}
Therefore, it remains to establish asymptotics of $R_t$.
\begin{lemma}\label{Ras}
Suppose that assumptions \eqref{ass1}-\eqref{ass4}, \eqref{ass5} are satisfied
and let $c_+,c_-$ be as in \eqref{righttail} and \eqref{lefttail}. Then  
\begin{equation}\label{rightR}
\lim _{t \to \8}\P ( R_t>t )t^{\a }(\log t)^{-\a  }=c_+\rho ^{\a }
\end{equation}
\begin{equation}\label{leftR}
\lim _{t \to \8}\P ( R_t<-t )t^{\a }(\log t)^{-\a  }=c_-\rho ^{\a }
\end{equation}
\end{lemma}

\begin{proof} In this proof, let $n=n_0-L-1$. Then $\lim _{n\to \8 }\frac{n}{\log t}=\rho $. Let
\begin{align*}
R'_t=&n\Pi _nS_{2L}\\
R''_t=& \sum _{m=1}^{2L}m\Pi _{n+m}b_{2,n+m+1}
\end{align*}
Then
$$
R_t=R'_t+R''_t.
$$
We are going to prove
\begin{align}
\lim _{t\to \8} &\P (R'_t>t)t^{\a }(\log t)^{-\a}= c_+ \rho ^{\a }\label{rprim}\\
\lim _{t\to \8} &\P (R''_t>t)=o\big (t^{-\a }(\log t)^{\a}\big ) \label{rbis}
\end{align}
\eqref{rprim} is equivalent to 
$$
\lim _{t\to \8} \P (\Pi _n S_{2L}>tn^{-1})t^{\a }(\log t)^{-\a}= c_+\rho ^{\a } 
$$
But by \eqref{righttail}
$$
\lim _{t\to \8} \P (\Xa _2>tn^{-1})t^{\a }(\log t)^{-\a}= c_+ \rho ^{\a }
$$
and
$$
\Xa _2=\sum _{m=1}^{n_0-L-1}\Pi _{m-1}b_{2,m}+\Pi _n S_{2L} +
\sum _{m=n_0+L+1}^{\8} \Pi _{m-1}b_{2,m}.
$$
So \eqref{rprim} follows from Lemma \ref{irrelev}.
For \eqref{rbis} let
$$
\tilde S_{2L}=\sum_{m=1}^{2L}a_{n+1}\cdots  a_{n+m}|b_{2,n+m+1}|.
$$
Then
$$
R''_t\leq 2L\Pi _n\tilde S_{2L}.
$$
and so \eqref{rbis} is equivalent to
$$
\lim _{t\to \8 }\P (\Pi _{n-1}\tilde S_{2L}>t(2L)^{-\a})t^{\a }(\log t)^{-\a}=0.$$
Let
\begin{equation*}
\tilde \Xa _2=\sum _{m=1}^{\8}\Pi _{m-1}|b_{2,m}|.
\end{equation*}
By \eqref{L} and Theorem \ref{kestengoldie},
$$
\lim _{t\to \8} \P (\tilde \Xa _2>t(2L)^{-1})t^{\a }(\log t)^{-\a}=0.
$$
But 
\begin{equation*}
\tilde \Xa _2=\sum _{m=1}^{n_0-L-1}\Pi _{m-1}|b_{2,m}|+\Pi _n\tilde S_{2L} +
\sum _{m=n_0+L+1}^{\8} \Pi _{m-1}|b_{2,m}|.
\end{equation*}
Hence \eqref{rbis} follows from Lemma \eqref{irrelev}.
\end{proof}

Finally, we obtain
\begin{cor}\label{nonzero}
Let $c_+,c_-$ be as in \eqref{righttail} and \eqref{lefttail}.
Suppose that the assumptions \eqref{ass2}-\eqref{ass4}, \eqref{ass6}, \eqref{ass5} are satisfied and there is $r\geq 3$, $r>\a $ such that $\E |y|^ra^{\a }<\8 $. Then 
$$
\lim _{t \to \8}\P (M_t>t)t^{\a }(\log t)^{-\a }=\begin{cases} c_+s^{\a }\rho ^{\a }\ \mbox{if}\ \ \ s>0\\ c_-|s|^{\a }\rho ^{\a } \ \mbox{if}\ \ \ s<0\end{cases}
$$
and
$$
\lim _{t \to \8}\P (M_t<-t)t^{\a }(\log t)^{-\a }=\begin{cases} c_-s^{\a }\rho ^{\a }\ \mbox{if}\ \ \ s>0\\ c_+|s|^{\a }\rho ^{\a } \ \mbox{if}\ \ \ s<0\end{cases}
$$
\end{cor}

\section{Appendix}
For the reader convenience we recall three theorems that are used in the proofs of Theorems \ref{mthm1} and \ref{mthm}. Define a Markov process  $\{ W_n\}$  on $\R$ by the formula
$$
W_n = M_nW_{n-1}+Q_n,\ \ n\geq 1,
$$
where $(M_n,Q_n)\in \R^+ \times \R $ is a sequence of i.i.d. 
random variables and $W_0\in \R$ is an initial distribution.
If $\E \log M < 0$ and $\E \log^+|Q|<\8$, the sequence $\{W_n\}$ converges in law to a random
variable $W$, which is the unique solution to the random
difference equation
$$
W =_d MW+Q, \qquad \mbox{$W$ independent
of }(A,B);
$$
see \cite{V}.  The following result of Kesten \cite{K} and Goldie \cite{G} describes the tail of $W$.
 \begin{thm}
 \label{kestengoldie}
 Assume that the law of $\log M$ is non-arithmetic, $\E \log M <0$, $\E M^\a = 1$ for some $\a>0$ and $\E[|Q|^\a + M^\a \log^+M]<\8$. Then
$$
\lim_{t\to\8}t^\a \P[W>t] = C_+ \qquad \mbox{and} \qquad \lim_{t\to\8} t^\a\P[W<-t] = C_-.
$$
Moreover,
   $ C_++C_->0$ if and only if
$$
\P[Mx+Q=x]<1 \mbox{ for every $x\in \R$} .
$$
 \end{thm}
To estimate the error in the central limit theorem we use the following theorem, \cite{Petrovbook}
 \begin{thm}
 \label{petrov1}
Let $Y_1,\dots Y_n$ be independent identically distributed random variables, $\E Y_1=0$, $\E Y_1^2=\s ^2>0$, $\E |Y_1|^r <\infty $ for some $r\geq 3$. Let
$$
F_n(x)=\P \big ( \s ^{-1}n^{-1\slash 2}\sum _{j=1}^nY_j<x\big ).
$$
and
$$
\Phi (x) =(2\pi )^{-1\slash 2}\int _{-\infty }^xe^{-t^2\slash 2}\ dt
$$
Then 
$$
|F_n (x)-\Phi (x)|\leq C(r)(1+|x|)^{-r}(\s ^{-3}\E |Y_1|^3n^{-1\slash 2} +\s ^{-r}\E |Y_1|^rn^{-(r-2)\slash 2})
$$
for all $x$, where $C(r)$ is a positive constant depending only on $r$.
\end{thm}

For a positive random variable $M$ let $\Lambda (\b )=\log \E M^{\b }$. Suppose that $\Lambda $ is well defined for $0\leq \b < \b _0 \leq \8 $. Then so is $\Lambda '$. Let $\l =\supp _{\b <\b_0}\Lambda '(\b )$. 
The following uniform large deviation theorem is due to \cite{Petrov}, Theorem 2.

\begin{thm}\label{petrov2} 
Suppose  that $c$ satisfies $ {\mathbb E} \left[ \log A \right] < c < \l $, and suppose that $\delta(n)$ is an arbitrary function satisfying  $\lim_{n \to \infty} \delta(n) = 0$.
Also, assume that the law of $\log \ M$ is non-lattice.
Then with 
$\alpha$ chosen such that $\Lambda '(\alpha )=c$, we have that
\begin{align*} 
{\mathbb P} & \left\{ \log M_1+\dots +\log M_n > n(c + \gamma_n) \right\}\nonumber\\
& \quad\quad = \frac{1}{\alpha\sigma(\alpha) \sqrt{2\pi n}} \exp\left\{-n \Big( \alpha(c+\gamma_n) - \Lambda(\alpha) + \frac{\gamma_n^2}{2\sigma^2(\alpha)}
  \left(1 + O(|\gamma_n| \right) \Big) \right\} (1+o(1))
\end{align*}
as $n \to \infty$,
uniformly with respect to $c$ and $\gamma_n$ in the range
\begin{equation} \label{petrov-0}
{\mathbb E} \left[ \log \ M \right] + \epsilon \le c \le \l - \epsilon \quad \text{\rm and} \quad |\gamma_n| \le \delta(n),
\end{equation}
where $\epsilon >0$.
\end{thm}

\begin{rem}{\rm
In \eqref{petrov-0}, we may have that $\sup \{ \b : \b \in dom (\Lambda) \} = \infty$ or ${\mathbb E} \left[ \log \,M \right] =-\infty$.  In these cases, the quantities
$\infty -\epsilon$ or  $-\infty -\epsilon$ should be interpreted as arbitrary positive, respectively negative, constants.}
\end{rem}

\subsection*{Acknowledgments}
The research was supported by the NCN under Grant DEC-2014/15/B/ST1/00060.

\end{document}